\documentclass[12pt]{amsart}
\usepackage{amsmath, amsthm, amsfonts, amssymb, color}
\usepackage{mathrsfs}
\usepackage{amsfonts, amsmath}
\usepackage{amsmath,amstext,amsthm,amssymb,amsxtra}
\usepackage{txfonts}
\usepackage{tikz}
\allowdisplaybreaks[4]
\usepackage[colorlinks, citecolor=blue,hypertexnames=false]{hyperref}
\usepackage{pgf,tikz}

\theoremstyle{plain}
\newtheorem{theorem}{Theorem}[section]
\newtheorem{prop}[theorem]{Proposition}
\newtheorem{lemma}[theorem]{Lemma}

\theoremstyle{definition}
\newtheorem{definition}[theorem]{Definition}

\def\RR{\Bbb R}
\def \R   {\mathbb{R}^n}
\def \RN   {\mathbb{R}^n}
\def \UHRN {{\mathbb R}^{n+1}_+}

\def \HMAX {H^p_{L,max}(\mathbb{R}^n)}

\allowdisplaybreaks

\title[A maximal function characterization for Hardy spaces ]
{A maximal function characterization for Hardy spaces\\ associated to nonnegative self-adjoint  operators\\ satisfying
Gaussian estimates}

\author{Liang Song \ and \ Lixin Yan  }
\address{
   Liang Song,
    Department of Mathematics,
    Sun Yat-sen University,
    Guangzhou, 510275,
    P.R.~China}
\email{ songl@mail.sysu.edu.cn}

\address{
    Lixin Yan,
    Department of Mathematics,
    Sun Yat-sen University,
    Guangzhou, 510275,
    P.R.~China}
\email{
    mcsylx@mail.sysu.edu.cn}



\begin{document}

\begin{abstract}
Let $L$ be
a nonnegative, self-adjoint operator satisfying Gaussian estimates
 on $L^2(\RR^n)$.  In this article we give an atomic decomposition for the Hardy spaces
 $ H^p_{L,max}(\R)$  in terms of the nontangential maximal functions associated with the heat semigroup
 of $L$,  and this leads   eventually to characterizations of Hardy spaces associated to $L$,  via
  atomic decomposition or the nontangential  maximal functions.
  The proofs  are based on a modification of technique due to A. Calder\'on \cite{C}.
\end{abstract}

\maketitle


\section{Introduction}\label{sec:intro}
\setcounter{equation}{0}

The introduction and development of  Hardy
 spaces on Euclidean spaces
 ${\mathbb R}^n$ in the 1960s    played an important role
 in modern harmonic analysis and applications in partial differential equations.
Let us recall
the definition of the   Hardy spaces  (see \cite{Co, FS, L, St2, SW}).
Consider the Laplace operator $\Delta=-\sum_{i=1}^n\partial_{x_i}^2$ on the Euclidean spaces $\mathbb R^n$.
For $0<p<\infty$, the   Hardy space
$H^p({\mathbb{R}}^n)$ is defined as the space of tempered distribution
$f\in{\mathscr S}'({\mathbb{R}}^n)$ for which
  the area integral function of
$  f$ satisfies
\begin{eqnarray}
Sf (x) :=\left( \int_0^{\infty}\!\!\!\!\int_{ |y-x|<t}
\big|t^2 \triangle
e^{-t^2{\triangle}}f(y)\big|^2\  {dy \ \! dt \over t^{n+1}}\right)^{1/2}
\label{e1.1}
\end{eqnarray}
belongs to $L^p({\mathbb{R}}^n)$. If this is the case, define
\begin{eqnarray}\label{e1.2}
\|f\|_{H^p({\mathbb{R}}^n)}:=\|Sf\|_{L^p({\mathbb{R}}^n)}.
\end{eqnarray}
When $p>1$, $H^p({\mathbb{R}}^n)=L^p({\mathbb{R}}^n)$. For $p\leq 1,$ the space $H^p({\mathbb{R}}^n)$  involves many different
characterizations. For example, if $f\in{\mathscr S}'({\mathbb{R}}^n)$, then
\begin{eqnarray}\label{e1.3}
f\in H^p({\mathbb{R}}^n)
& \substack{{\rm (i)}\\ \Longleftrightarrow}& \sup_{t>0}\big|e^{-t^2\Delta}f(x)\big|\in L^p({\mathbb{R}}^n)\nonumber\\
&\substack{{\rm (ii)}\\ \Longleftrightarrow}& \sup_{|y-x|<t}\big|e^{-t^2\Delta}f(y)\big|\in L^p({\mathbb{R}}^n)\nonumber\\
&\substack{{\rm (iii)}\\  \Longleftrightarrow}&
 f {\rm \ has\ a \  {\it (p, q)}  \ atomic \ decomposition}\
 f=\sum_{j=0}^{\infty}\lambda_j a_j  \ {\rm with\ }   \sum_{j=0}^{\infty}|\lambda_j|^p<\infty.
\end{eqnarray}
  Recall that a   function $a$ supported in ball $B$ of ${\mathbb R}^n$
  is called
  a  {\it $(p, q)$}-atom, $0<p\leq 1\leq q\leq \infty$, $p<q$,  if
  $\|a\|_{L^q(B)}\leq |B|^{{1\over q}-{1\over p}}$,
and  $\int_B x^\alpha a(x)dx=0,$ where $\alpha$ is a multi-index of order $|\alpha|\leq  \left[n({1/p}-1)\right],$
the integer part of $ n({1/p}-1)$ (see \cite{Co, L, St2}).

 The theory of classical Hardy spaces has been very successful and fruitful
in the  past decades.
 However,
there are important situations in which
the standard theory of Hardy spaces    is not applicable, including certain problems in
  the theory of
 partial differential equation which involves  generalizing  the  Laplacian.
 There is a need to consider Hardy spaces that are  adapted to a linear
 operator $L$, similarly to the way that the standard theory of Hardy spaces  are adapted to the Laplacian.
 This topic  has attracted a lot of attention in the last   decades,
 and  has been a  very active research topic  in harmonic analysis -- see for example,
    \cite{ADM, AMR, BeZ,   DL, DY, DZ,  HLMMY, HM, HMMc, JY,  Y}.

In this article, we assume that  $L$ is a densely-defined operator
on $L^2(\RR^n) $ and satisfies  the following properties:

\smallskip

  \noindent
 {\bf   (H1)} \
  $L$ is a second order non-negative self-adjoint operator on $L^2(\RR^n)$;

\smallskip

 \noindent
 {\bf   (H2)} \ The kernel of $e^{-tL}$, denoted by $p_t(x,y)$,
 is a measurable function on $\RR^n\times \RR^n$ and  satisfies
a Gaussian upper bound, that is
$$ \leqno{\rm (GE)}\hspace{4cm}
\big|p_t(x,y)\big|\leq C t^{-n/2} \exp\Big(-{  {|x-y|^2}\over  ct}\Big)
$$
for all $t>0$,  and $x,y\in \RR^n,$   where $C$ and $c$   are positive
constants.

 Given  a function $f\in L^2(\RR^n)$, consider the following
 area function   $S_{L}f$   associated to the heat semigroup
generated by $L$
\begin{equation}
S_{L}f(x):=\left(\int_0^{\infty}\!\!\int_{|x-y|<t}
|t^2Le^{-t^2L} f(y)|^2 {dydt\over   t^{n+1}}\right)^{1/2},
\quad x\in \RR^n.
\label{e1.5}
\end{equation}
Under the assumptions  ${\bf (H1)}$ and ${\bf (H2)}$ of an operator $L$,
it is known (see for example, \cite{Au, ADM}) that the function   $S_{L}$  is   bounded on $L^p({\mathbb R}^n ), 1<p<\infty$
and
\begin{equation}
  \|S_Lf\|_{L^p(\RR^n)} \simeq \|f\|_{L^p(\RR^n)}.
\label{e1.6}
\end{equation}

\begin{definition}\label{def1.1}
  Suppose  that an operator $L$ satisfies  ${\bf (H1)}$-${\bf (H2)}$.
 Given $0<p\leq 1$. The Hardy space $H^p_{L,  S}(\RR^n) $ is defined as the completion of
 $\{ f\in L^2(\RR^n):\, \|S_{L} f\|_{L^p(\RR^n)}<\infty \}
 $
 with  norm
$$
 \|f\|_{H^p_{L, S}(\RR^n)}:=\|S_{L} f\|_{L^p(\RR^n)}.
$$
\end{definition}

\smallskip

To  describe  an atomic  character  of
the   Hardy spaces, let us recall the notion of
 $(p,q,M)$-atom associated to
an operator $L$ (\cite{DL, HLMMY}).

\begin{definition}\label{def1.2}
Given   $0<p\leq 1\leq q\leq \infty$, $p<q$  and   $M\in {\mathbb N}$,
a function $a\in L^2(\RR^n)$ is called a $(p,q,M)$-atom associated to
the operator $L$ if there exist a function $b\in {\mathcal D}(L^M)$
and a ball $B\subset \RR^n$ such that

\smallskip

{  (i)}\ $a=L^M b$;

\smallskip

{  (ii)}\ {\rm supp}\  $L^{k}b\subset B, \ k=0, 1, \dots, M$;

\smallskip

{  (iii)}\ $\|(r_B^2L)^{k}b\|_{L^q (\RR^n)}\leq r_B^{2M}
|B|^{1/q-1/p},\ k=0,1,\dots,M$.
\end{definition}

\smallskip

The  atomic Hardy space $H^p_{L, {\rm at}, q, M}(\R)
 $ is defined as follows.

\begin{definition} \label{def1.3}
 We will say that $f= \sum\lambda_j
a_j$ is an atomic $(p, q, M)$-representation (of $f$) if $
\{\lambda_j\}_{j=0}^{\infty}\in {\ell}^p$, each $a_j$ is a
$(p, q, M)$-atom, and the sum converges in $L^2(\R).$  Set
\begin{eqnarray*}
\mathbb  H^p_{L, {\rm at}, q, M}(\R):=\Big\{f:  f \mbox{  has
an atomic $(p, q, M)$-representation} \Big\},
\end{eqnarray*}
with the norm $\big\|f\big\|_{\mathbb  H^p_{L, {\rm at}, q, M}(\R)}$ given by
\begin{eqnarray*}
\inf\Big\{\Big(\sum_{j=0}^{\infty}|\lambda_j|^p\Big)^{1/p}:
f=\sum\limits_{j=0}^{\infty}\lambda_ja_j \ \mbox{ is an atomic
$(p,q,M)$-representation}\Big\}.
\end{eqnarray*}
The space $H^p_{L, {\rm at}, q, M}(\R)$ is then defined as
the completion of $\mathbb  H^p_{L, {\rm at}, q, M}(\RR^n)$ with
respect to this norm.
 \end{definition}

Obviously,
$H^p_{L, {\rm at}, q_2, M}(\R)\subseteq H^p_{L, {\rm at}, q_1, M}(\R)
$  when $1<q_1\leq q_2\leq \infty$.
 Under the assumption  that  an operator $L$ satisfies conditions    ${\bf (H1)}$-${\bf (H2)}$,
   S. Hofmann, G. Lu, D. Mitrea,
 M. Mitrea and the second named author of this article    obtained
 a  $(1, 2, M)$-atomic decomposition of
 the   Hardy space $H^1_{L,S}({\mathbb R^n})$, and showed that
 for every number $M> 1$,
the spaces $ H^1_{L,  S}(\RR^n)$ and $H^1_{L, {\rm at}, 2, M}(\RR^n)$ coincide (see \cite{HLMMY}). In particular,
$$
\|f\|_{ H^1_{L,  S}(\RR^n)}\approx\| f\|_{H^1_{L, {\rm at}, 2, M}(\RR^n)}.
$$
A proof for $p < 1$ was shown by Duong and Li in \cite{DL},
and by Jiang and Yang in \cite{JY2}.

Given a function $f\in L^2({\Bbb R}^n)$, consider the non-tangential maximal
function associated to  the heat semigroup generated by the
 operator $L$,
\begin{align*}
f^*_L(x):=\sup\limits_{|y-x|<t}|e^{-t^2L}f(y)|.
\end{align*}
 We may define the spaces $H^p_{L,max}(\R), 0<p\leq 1 $
as the completion of $\{f \in L^2(\R): \|f^*_L\|_{L^p(\R)}<
\infty\}$ with respect to $L^p$-norm of the non-tangential maximal function; i.e.,
\begin{eqnarray*}
\big\|f\big\|_{H^p_{L,max}(\R)}:=\big\|f^*_L\big\|_{L^p(\R)}.
\end{eqnarray*}
It can be verified (see \cite{HLMMY, DL}) that for every $1<q\leq \infty$ and every
number $M>\frac{n}{2}(\frac{1}{p}-1)$, any $(p,q,M)$-atom $a$ is in $ H^p_{L,max}(\R)$
  and so the following
continuous inclusions hold:
$$
H^p_{L, {\rm at}, q, M}(\RR^n)\subseteq  H^p_{L,max}(\R).
$$

A natural question  is to show   the following
continuous inclusion:   $ H^p_{L,max}(\R)\subseteq
H^p_{L, {\rm at}, q, M}(\RR^n)$.
It is known that
   the inclusion $ H^p_{L,max}(\R)\subseteq
H^p_{L, {\rm at}, q, M}(\RR^n)$ holds  for some operators including
 Schr\"odinger operators   with nonnegative potentials  (see for example, \cite{DZ, DL, HLMMY}).
However,
 this question is still open  assuming merely  that an operator $L$ satisfies  ${\bf (H1)}$-${\bf (H2)}$.
 The aim of this article  is to give an affirmative answer to this
question to  get an
atomic decomposition  directly  from $ H^p_{L,max}(\R)$. We have  the following result.
 \begin{theorem}\label{th1.1}  Suppose  that an operator $L$ satisfies  ${\bf (H1)}$-${\bf (H2)}$.
Fix $0<p\leq 1$ and $M> {n\over 2}({1\over p}-1)$. Then the  following three conditions are equivalent:
\begin{itemize}
\item[(i)] $f\in H^p_{L, {\rm at},   \infty, M}(\R);$ \\[1pt]

\item[(ii)] Given    $\alpha>0$,
$\varphi^{\ast}_{L, \alpha}f=\sup\limits_{|y-x|<\alpha t}|\varphi(t\sqrt{L})f(y)| \in L^p(\RR^n)$
for some even  function $\varphi\in{\mathscr S}(\RR)$, $ \varphi(0)=1;$\\[1pt]

\item[(iii)]  $G^{\ast}_{L}(f)= \sup\limits_{\varphi\in {\mathscr A}}
\sup\limits_{|y-x|<  t}|\varphi(t\sqrt{L})f(y)|\in L^p(\RR^n)$,
where
$$
{\mathscr A}=\left\{ \varphi\in{\mathscr S}(\RR):   {\rm even \ functions \ with \ }  \varphi(0)\not=0,
\int_{\RR}(1+|x|)^{N} \sum_{k\leq N} \Big|{d^k\over dx^k} \varphi(x)\Big|^2dx\leq 1\right\},
$$
where $N$ is a large number depending only on $p$ and $n$.
 \end{itemize}
\end{theorem}

 \medskip
 We should mention that using the theory of tent spaces,
 a  $(p, 2, M)$-atomic decomposition of
 the   Hardy space $H^p_{L,S}({\mathbb R^n})$ in terms of area  functions was  given
 in \cite{DL, HLMMY}.
In this article, we shall use    a different argument to build a $(p, \infty, M)$-atomic decomposition  of
 the   Hardy spaces $H^p_{L, max}({\mathbb R^n})$ in terms of maximal functions.
Our proof is  based on a modification of technique due to A. Calder\'on \cite{C},
 where a decomposition  of the function $ F(x,t)=f\ast \varphi_t(x)$ associated with the distribution $f$
 was given, and convolution operation of the function $F$ played an important role in the proof.
  In our setting,
  there is, however,
  no analogue of
convolution operation of the function  $ t^2L e^{-t^2L}f(x)$, we have to modify Calder\'on's
construction and
 the geometry   is conducting the
analysis (see Figure 1 in Section 3).
On the other hand, we do not assume that the heat kernel $p_t(x,y)$
	  satisfy the standard regularity condition, thus standard techniques of
Calder\'on--Zygmund theory (\cite{CT, St2}) are not applicable. The
lacking of smoothness of the kernel
will be overcome in Proposition~\ref{le3.1} below by using
some estimates on heat kernel bounds,
finite propagation speed of solutions to the wave equations and
spectral theory of non-negative self-adjoint operators.

Throughout, the letter ``$c$"  and ``$C$" will denote (possibly
different) constants  that are independent of the essential
variables.

\medskip

\bigskip

\section{Preliminaries}
\setcounter{equation}{0}

Recall that,  if $L$ is a nonnegative,
self-adjoint operator on $L^2({\Bbb R}^n)$, and $E_L(\lambda)$ denotes a spectral
decomposition associated with $L$, then for every bounded Borel function
$F:[0,\infty)\to{\Bbb C}$, one defines the operator
$F(L): L^2({\Bbb R}^n)\to L^2({\Bbb R}^n)$ by the formula
\begin{align}\label{e2.1}
F(L):=\int_0^{\infty}F(\lambda) \, {\rm d}E_L(\lambda).
\end{align}
 In particular,
the  operator $ \cos(t\sqrt{L})$  is then well-defined  on
$L^2({\Bbb R}^n)$. Moreover, it follows from Theorem 3.4 of
\cite{CS}
  that
 the integral kernel
 $K_{\cos(t\sqrt{L})}$ of $\cos(t\sqrt{L})$ satisfies
\begin{align}
 {\rm supp}K_{\cos(t\sqrt{L})}\subseteq
 \bigl\{(x,y)\in {\Bbb R}^n\times {\Bbb R}^n: |x-y|\leq   t\bigr\}.
 \label{e2.2}
 \end{align}
  By the Fourier inversion formula, whenever $F$ is
an even bounded Borel function with the Fourier transform of $F$,
$\hat{F} \in L^1(\mathbb{R})$, we can  write $F(\sqrt{L})$ in terms
of $\cos(t\sqrt{L})$. Concretely, by recalling (\ref{e2.1}) we have
$$
F(\sqrt{L})=(2\pi)^{-1}\int_{-\infty}^{\infty}{\hat
F}(t)\cos(t\sqrt{L}) \, {\rm d}t,
$$
which, when combined with (\ref{e2.2}), gives
\begin{align}\label{e2.3}
K_{F(\sqrt{L})}(x,y)=(2\pi)^{-1}\int_{|t|\geq  |x-y|}{\hat F}(t)
K_{\cos(t\sqrt{L})} (x,y) \, {\rm d}t.
\end{align}

\begin{lemma}\label{le2.1}\,  Let $\varphi\in C^{\infty}_0(\mathbb R)$ be
even, $\mbox{supp}\,\varphi \subset (-1, 1)$. Let $\Phi$ denote the Fourier transform of
$\varphi$. Then for every $\kappa=0,1,2,\dots$, and for every $t>0$,
the kernel $K_{(t^2L)^{\kappa}\Phi(t\sqrt{L})}(x,y)$ of the operator
$(t^2L)^{\kappa}\Phi(t\sqrt{L})$ which was defined by the spectral theory, satisfies
\begin{eqnarray}\label{e2.51}
{\rm supp}\ \! K_{(t^2L)^{\kappa}\Phi(t\sqrt{L})}
\subseteq \big\{(x,y)\in \RN\times \RN: |x-y|\leq t\big\}
\end{eqnarray}
 and
\begin{eqnarray}\label{e2.61}
\big|K_{(t^2L)^{\kappa}\Phi(t\sqrt{L})}(x,y)\big|
\leq C  \, t^{-n}
\end{eqnarray}
for all  $x,y\in \RN.$
\end{lemma}

\smallskip

\begin{proof}  For the proof,  we refer it to \cite{GY} and \cite{HLMMY}.
\end{proof}

\begin{lemma}\label{le2.2}    Assume that an operator $L$ satisfies  ${\bf (H1)}$-${\bf (H2)}$.
 Let  $R>0, s>0$.
Then for any $\epsilon>0$, there exists a constant
$C=C(s, \epsilon)$ such that
\begin{eqnarray*}
\int_{\RR^n} \big|K_{F(\sqrt {L})}(x,y)\big|^2 \big(1+R|x-y|\big)^{s} dy\leq
CR^n
 \|  F(R\cdot)\|^2_{C^{{s\over 2} +\epsilon}(\mathbb R)}
\end{eqnarray*}
for all Borel functions $F$ such that supp $F\subseteq [0, R].$
\end{lemma}

\begin{proof}
For the proof, we refer the reader to Lemma 7.18, \cite{O}. See also \cite{DOS}.
\end{proof}

Next we show the following result, which will be useful in the sequel.

\begin{lemma}\label{prop2.3} Assume that an operator $L$ satisfies  ${\bf (H1)}$-${\bf (H2)}$.
Let $\psi_i \in {\mathscr S}(\mathbb R) $ be even functions, $\psi_i(0)=0, i=1,2$.
   Then for every $\eta>0$, there exists a positive constant $C=C({n, \eta, \psi_1, \psi_2})$
such that the kernel $K_{\psi_1(s\sqrt{L}) \psi_2(t\sqrt{L}) }(x,y)$ of $ \psi_1(s\sqrt{L}) \psi_2(t\sqrt{L})$
satisfies
\begin{eqnarray}\label{e2.5}
\big|K_{\psi_1(s\sqrt{L}) \psi_2(t\sqrt{L}) }(x,y)\big|
\leq  C\, \left({\min(s,t)\over \max(s,t)}\right)
  {\max(s,t)^\eta\over (\max(s,t)+ |x-y|)^{n+\eta}}
\end{eqnarray}
for all $t>0$ and $x, y\in \RN$.
\end{lemma}

\begin{proof}  By   symmetry, it suffices to show that if $s\leq t$, then
\begin{eqnarray}\label{e2.6}
\big|K_{\psi_1(s\sqrt{L}) \psi_2(t\sqrt{L}) }(x,y)\big|
\leq  C\, \left({s\over t}\right)
  {t^\eta\over (t+ |x-y|)^{n+\eta}}.
\end{eqnarray}
To do this, we fix $s, t>0$ and let $\Psi(tx)={t\over s} \psi_1(sx)  \psi_2(tx)$,
and so $
\psi_1(s\sqrt{L}) \psi_2(t\sqrt{L}) ={s\over t} \Psi(t\sqrt{L}).
$
 Let us show that
\begin{eqnarray}\label{e2.7}
\big|K_{\Psi(t \sqrt{L})}(x,y)\big|
\leq  Ct^{-n}, \ \ \ \  x,y\in \RN.
\end{eqnarray}
Indeed,  for any $\kappa\in {\Bbb N}$, we have   the relationship
\begin{eqnarray}\label{e2.8}
  (I+ t^2L)^{-\kappa}={1\over  (\kappa-1)!} \int\limits_{0}^{\infty}e^{- ut^2L}e^{-u} u^{\kappa-1} du
\end{eqnarray}
  and so when $\kappa>n/4$,
 \begin{eqnarray*}
 \big\| (I+ t^2L)^{-\kappa} \big\|_{L^2(\RR^n)\rightarrow L^{\infty}(\RR^n)}\leq {1\over  (\kappa-1)!} \int\limits_{0}^{\infty}
 \big\| e^{- ut^2L}\big\|_{L^2(\RR^n)\rightarrow L^{\infty}(\RR^n)} e^{-u} u^{\kappa-1} du\leq Ct^{-n/2}.
 \end{eqnarray*}
Now
$ \big\| (I+ t^2L)^{-\kappa} \big\|_{L^1(\RR^n)\rightarrow L^{2}(\RR^n)}=
\big\| (I+ t^2L)^{-\kappa} \big\|_{L^2(\RR^n)\rightarrow L^{\infty}(\RR^n)}\leq
Ct^{-n/2}  $, and so
 \begin{eqnarray*}
 \big\|\Psi( t\sqrt{L})\big\|_{L^1(\R)\rightarrow L^{\infty}(\R)} \leq
 \big\| (I+ t^2L)^{2\kappa}\Psi( t\sqrt{L}) \big\|_{L^2(\RR^n)\rightarrow L^2(\RR^n)}
 \big\| (I+ t^2L)^{-\kappa} \big\|^2_{L^2(\RR^n)\rightarrow L^{\infty}(\RR^n)}.
 \end{eqnarray*}
  Since $\psi_1\in {\mathscr S}(\RR)$ and $\psi_1(0)=0$, we have that
  $(s\lambda)^{-1}\psi_1(s\lambda)=\int_0^1\psi'_1(s\lambda y) dy \in L^{\infty}(\RR)$, and then
the $L^2$ operator norm of the last term is equal to the $L^{\infty}(\RR)$
norm of the function $(1+ t^2|\lambda|)^{2m}  \Psi( t\sqrt{|\lambda|})=
[(s\sqrt{|\lambda|})^{-1}\psi_1(s\sqrt{|\lambda|})] [(1+ t^2|\lambda|)^{2m} (t\sqrt{|\lambda|})\psi_2(t\sqrt{|\lambda|})]
$ which is uniformly
bounded in $t>0$. This implies that (\ref{e2.7}) holds.

Next, we   write   $F (t\lambda)=\Psi(t\lambda)(1+t^2\lambda^2)^m$, where $m>n/2$. Then we have
 $\Psi( t\sqrt{L})=F(  t\sqrt{L})(1+ t^2L)^{-m} $.
From \eqref{e2.8}, it can be verified
that  for $m>n/2$, there exist some positive constants $C$ and $c$ such  that for every $t>0$,
the kernel $K_{(1+t^2L)^{-m}}(x,y)$ of the operator
$(1+t^2L)^{-m}$  satisfies
\begin{eqnarray*}
\big|K_{(1+t^2L)^{-m}}(x,y)\big|
\leq \frac{C}{t^{n} } \exp\Big(-{|x-y| \over
c\,t }\Big),
\end{eqnarray*}
which, in combination with
 $\big(1+{|x-y|\over t} \big) \leq  (1+{|x-z|\over t} )(1+{|y-z|\over t} )$, shows
\begin{eqnarray*}
 \Big|\Big(1+{|x-y|\over t }\Big)^{n+\eta}K_{\Psi(t\sqrt{L})}(x,y)\Big|
&&=  \Big(1+{|x-y| \over t}\Big)^{n+\eta}\Big|\int_{\RN}
K_{F( t\sqrt{L}) }(x,z) K_{ (1+ t^2L)^{-m}}(z,y)\, dz\Big|\\
&&\leq Ct^{-n}\int_{\RN}\big|K_{F( t\sqrt{L}) }(x,z)\big|\Big(1+{|x-z|\over t }\Big)^{n+\eta} dz.
\end{eqnarray*}
By symmetry, estimate \eqref{e2.6} will be proved  if we show that
\begin{eqnarray}\label{e2.9}
\int_{\RN}\big|K_{F(t\sqrt{L}) }(x,z)\big|\Big(1+{|x-z|\over t}\Big)^{n+\eta} dz\leq C.
\end{eqnarray}

Let   $\varphi\in C_c^{\infty}(0, \infty)$ be a non-negative function satisfying
    supp $\varphi \subseteq [{1\over 4}, 1]$ and   let $\varphi_0=1- \sum_{\ell=1}^{+\infty}\varphi(2^{-\ell}\lambda).$
 So,
   \begin{eqnarray*}
\varphi_0(\lambda)+\sum_{\ell=1}^{\infty}\varphi(2^{-\ell}\lambda) =1, \ \ \  \ \forall\,  \lambda >0.
\end{eqnarray*}
Let  $F^{0}(t\lambda)$ denote the function $\varphi_0(t\lambda)F(t\lambda)$  and for $\ell\geq 1 $
$
F^{\ell} (t\lambda):= \varphi(2^{-\ell}t\lambda)F(t\lambda).
$
From \eqref{e2.7}, the proof of \eqref{e2.9} reduces to estimate the following:
\begin{eqnarray}\label{e2.10}
\int_{\RN}\big|K_{F(t\sqrt{L}) }(x,z)\big|\Big(1+{|x-z|\over t}\Big)^{n+\eta} dz
&\leq & C+
 \int_{\RN}\big|K_{F^{0} (t\sqrt{L}) }(x,z)\big|\Big(1+{|x-z|\over t}\Big)^{n+\eta} dz\nonumber\\
&+ &
\sum_{\ell=1}^{\infty}\int_{|x-z|\geq t}\big|K_{F^{\ell} (t\sqrt{L}) }(x,z)\big|
 \Big({|x-z|\over t}\Big)^{n+\eta} dz\nonumber\\
&=:& C+ \sum_{\ell=0}^{\infty}I_{\ell}.
\end{eqnarray}
By Lemma~\ref{le2.2},
\begin{eqnarray*}
 I_0
& \leq& C t^{ n/2}\Big(\int_{\RN}\big|K_{F^{0} (t\sqrt{L}) }(x,z)
\big|^2\Big(1+{|x-z|\over t}\Big)^{3n+2\eta+1} dz\Big)^{1/2} \nonumber\\
&\leq &  C\|\delta_{1/t}F^{0}(t\cdot)\|_{C^{{3n\over 2}+2\eta+1}}.
\end{eqnarray*}
 Since $\psi_1\in {\mathscr S}(\RR)$ and $\psi_1(0)=0$, we have that
  $(s\lambda)^{-1}\psi_1(s\lambda)=\int_0^1\psi'_1(s\lambda y) dy \in {\mathscr S}(\RR)$. Then we have
\begin{eqnarray}\label{e2.11}
 I_0&\leq &
C\|\varphi_0(\lambda) \Psi(\lambda)(1+\lambda^2)^m\|_{C^{{3n\over 2}+2\eta+1}}\nonumber\\
&=&\|\varphi_0(\lambda) \int_0^1 \psi'_1(s\lambda y/t) dy [\lambda \psi_2(\lambda)(1+\lambda^2)^m]\|_{C^{{3n\over 2}+2\eta+1}}
\leq C.
\end{eqnarray}
For the term $I_{\ell}$, we use   Lemma~\ref{le2.2} again to   obtain
\begin{eqnarray*}
  I_{\ell}
& \leq& Ct^{ n/2}\Big( \int_{\RN}\big|K_{F^{\ell}(t\sqrt{L}) }(x,z)\big|^2
\Big({|x-z|\over t}\Big)^{3n+2\eta+1} dz\Big)^{1/2} \nonumber\\
  & \leq&Ct^{ n/2}
  2^{-\ell(3n+2\eta+1)/2}\Big( \int_{\RN}\big|K_{F^{\ell}(t\sqrt{L}) }(x,z)\big|^2
 \Big(1+{2^{\ell}|x-z|\over t}\Big)^{3n+2\eta+1} dz\Big)^{1/2} \nonumber\\
&\leq & C2^{-\ell(3n+2\eta+1)/2}2^{\ell n/2} \|\delta_{2^{\ell}/t}
F^{\ell}(t\cdot)\|_{C^{{3n\over 2}+2\eta+1}}.
\end{eqnarray*}
It can be verified that for $\psi_i\in {\mathscr S}(\RR), i=1,2$,
 \begin{eqnarray*}
 \|\varphi \delta_{2^{\ell}}F\|_{C^{{3n\over 2}+2\eta+1}}
&=&\|\varphi(\lambda) \int_0^1 \psi'_1(2^{\ell}s\lambda y/t) dy
[2^{\ell}\lambda \psi_2(2^{\ell}\lambda)(1+2^{2\ell}\lambda^2)^m]\|_{C^{{3n\over 2}+2\eta+1}}\\
& \leq&    C2^{ \ell({3\over 2}n +2\eta+1)} 2^{ -2n\ell},
\end{eqnarray*}
which gives
 \begin{eqnarray}\label{e2.12}
  \sum_{\ell=1}^{\infty}I_{\ell}
& \leq&   C \sum_{\ell=1}^{\infty} C2^{-\ell(3n+2\eta+1)/2}2^{\ell n/2}2^{ \ell({3\over 2}n +2\eta+1)} 2^{ -2n\ell}
\nonumber\\
&\leq & C \sum_{\ell=1}^{\infty} 2^{-n\ell }
 \leq   C .
\end{eqnarray}
Putting \eqref{e2.11} and \eqref{e2.12} into \eqref{e2.10},
estimate \eqref{e2.9} follows readily. The proof of Lemma~\ref{prop2.3} is complete.
\end{proof}

\medskip

\section{Proof of Theorem~\ref{th1.1}}
\setcounter{equation}{0}

The proof of Theorem~\ref{th1.1} follows the line of  (ii)$\Rightarrow$ (i) $\Rightarrow$ (iii)$\Rightarrow$ (ii).
The proof of    (i) $\Rightarrow$ (iii) will be  an adaptation of the proof of the earlier
known implication  of (i) $\Rightarrow$ (ii) (see \cite{HLMMY, DL}). Obviously,  (iii)$\Rightarrow$ (ii).
  The left of the proof of Theorem~\ref{th1.1}
is to show an implication  (ii) $\Rightarrow$ (i). To do this,   we first show the following result.

\begin{prop}\label{le3.1}
Let $0<p\leq 1$. Let  $L$  be  a non-negative self-adjoint operator   on $L^2(\RR^n)$ satisfying
Gaussian  estimates ${\rm (GE)}.$
Let $\varphi_i \in {\mathscr S}(\mathbb R)$ be even functions with $ \varphi_i(0)=1$
and $\alpha_i>0, i=1,2$.
Then there exists a constant $C=C(n, \varphi_1, \varphi_2, \alpha_1, \alpha_2)$ such that
for every $f\in L^2(\RR^n)$, the functions
 $\varphi^{\ast}_{i, L, \alpha}f=\sup\limits_{|y-x|<\alpha t}|\varphi_i(t\sqrt{L})f(y)|, i=1,2$,
 satisfy
\begin{eqnarray}\label{e4.1}
\left\|\varphi^{\ast}_{1, L, \alpha_1}f\right\|_{L^p(\RR^n)}\leq
C\left\|\varphi^{\ast}_{2, L, \alpha_2}f\right\|_{L^p(\RR^n)}.
\end{eqnarray}
As a consequence, for any  $\varphi \in {\mathscr S}(\mathbb R)$ be even function with $\varphi(0)=1$,
$$
 C^{-1}\|f^{\ast}_L\|_{L^p(\RR^n)}\leq  \left\|\varphi^{\ast}_{L, \alpha}f\right\|_{L^p(\RR^n)}\leq
 C\|f^{\ast}_L\|_{L^p(\RR^n)}, \ \ \  \alpha>0.
 $$
\end{prop}

\begin{proof}  Recall  that
for any $0<\alpha_2\leq \alpha_1$,
$$\left\|\varphi^{\ast}_{L, \alpha_1}f\right\|_{L^p(\RR^n)}
\leq C\left(1+{\alpha_1\over\alpha_2}\right)^{n/p}
\left\|\varphi^{\ast}_{L, \alpha_2}f\right\|_{L^p(\RR^n)}
$$
for any  $\varphi \in {\mathscr S}(\mathbb R)$ (Theorem 2.3, \cite{CT}).
Now, we let $\psi(x):=\varphi_1(x)-\varphi_2(x)$, and then the proof of \eqref{e4.1} reduces to show that
\begin{eqnarray}\label{e4.2}
\left\|\psi^{\ast}_{L, 1}f\right\|_{L^p(\RR^n)}\leq
C\left\|\varphi^{\ast}_{2, L, 1}f\right\|_{L^p(\RR^n)}.
\end{eqnarray}

Let us show \eqref{e4.2}. Let $\Psi(x)=x^{2\kappa}\Phi(x)$ where
$\Phi(x)$ is the  function as in Lemma \ref{le2.1} and $2\kappa>{(n+1)/p}$.
By the  spectral theory (\cite{Y}), we have
$$
f=C_{\Psi, \  \varphi_2}\int_0^\infty \Psi(s\sqrt{L})\varphi_2(s \sqrt{L})f \, \frac{ds}{s}.
$$
Therefore,
$$
\psi(t\sqrt{L})f(x)=C \int_0^\infty \left(\psi(t\sqrt{L})\Psi(s \sqrt{L})\right)\varphi_2(s \sqrt{L})f(x) \, \frac{ds}{s}.
$$
 Let us denote the kernel of $\psi(t\sqrt{L})\Psi(s\sqrt{L})$ by
 $ K_{\psi(t\sqrt{L})\Psi(s\sqrt{L})}(x,y)$. Then for    $\lambda\in (\frac{n}{p}, \ 2\kappa)$,
\begin{align}\label{e4.3}
\hspace{-0.8cm}&\sup\limits_{|w|< t}|\psi(t\sqrt{L})f(x-w)|\\
&=C\sup\limits_{|w|< t}\big|\int_{\UHRN} K_{\psi(t\sqrt{L})\Psi(s\sqrt{L})}(x-w,z)
\varphi_2(s\sqrt{L})f(z) \frac{dzds}{s}\Big|\nonumber\\
&\leq C\sup\limits_{|w|< t}\int_{\UHRN} \big|K_{\psi(t\sqrt{L})\Psi(s\sqrt{L})}(x-w,z)
\big|\Big(1+\frac{|x-z|}{s}\Big)^{\lambda} \big|\varphi_2(s\sqrt{L})f(z)
\big|\Big(1+\frac{|x-z|}{s}\Big)^{-\lambda} \frac{dzds}{s}\nonumber\\
&\leq  \sup\limits_{z,s} \big|\varphi_2(s\sqrt{L})f(z)\big|\Big(1+\frac{|x-z|}{s}\Big)^{-\lambda}
 \sup\limits_{|w|< t}\int_{\UHRN} \big|K_{\psi(t\sqrt{L})\Psi(s\sqrt{L})}(x-w,z)
 \big|\Big(1+\frac{|x-z|}{s}\Big)^{\lambda}\, \frac{dzds}{s}.\nonumber
\end{align}
Next we will prove that
\begin{eqnarray}\label{e5}
\sup\limits_{|w|< t} \int_{\UHRN} \big|K_{\psi(t\sqrt{L})\Psi(s\sqrt{L})}(x-w,z)
\big|\Big(1+\frac{|x-z|}{s}\Big)^{\lambda}\, \frac{dzds}{s}\leq C.
\end{eqnarray}
Once  estimate \eqref{e5} is shown, \eqref{e4.2} follows. Indeed, it follows from
\eqref{e4.3}, \eqref{e5} and the condition $\lambda\in (\frac{n}{p}, \ 2\kappa)$  that
\begin{eqnarray*}
\left\|\psi^{\ast}_{L, 1}f\right\|_{L^p(\RR^n)}=\Big\|\sup\limits_{|w|< t}|\psi(t\sqrt{L})f(x-w)|\Big\|_{L^p_x(\R)}&\leq&
C\Big\|\sup\limits_{z,s} \big|\varphi_2(s\sqrt{L})f(z)\big|\Big(1+\frac{|x-z|}{s}\Big)^{-\lambda}\Big\|_{L^p_x(\R)}\nonumber\\
&\leq& C\left\|\sup\limits_{|y-x|<  t}|\varphi_2(t\sqrt{L})f(y)|\right\|_{L^p_x(\R)}\\
&= &C\Big\|\varphi^{\ast}_{2, L, 1}f\Big\|_{L^p(\RR^n)},
\end{eqnarray*}
where  we used  Theorem 2.4  of \cite{CT} in the second inequality.

 Let us prove \eqref{e5}.
Note that $|w|<t$. We write
\begin{eqnarray*}
\psi(t\sqrt{L})\Psi(s\sqrt{L})
=\left
\{
\begin{array}{ll}
\left({s\over t}\right)^{2\kappa} [\psi(t\sqrt{L})(t\sqrt{L})^{2\kappa}\Phi(s\sqrt{L})],  \ \ \ {\rm if}\ \ \ s\leq t; \\[8pt]
\left({t\over s}\right)^{2} [(t\sqrt{L})^{-2}\psi(t\sqrt{L})  (s\sqrt{L})^{2\kappa+2}\Phi(s\sqrt{L})],  \ \ \ {\rm if}\ \ \ s> t.
\end{array}
\right.
\end{eqnarray*}
We then   apply Lemma~\ref{prop2.3} to obtain that  for  $\eta\in (\lambda, 2\kappa)$,
 \begin{align*}
 \big|K_{\psi(t\sqrt{L})\Psi(s\sqrt{L})}(x-w,z)\big|\leq C
 \min\left( \left({s\over t}\right)^{2\kappa}, \left({t\over s}\right)^{2}\right)
  {\max(s,t)^\eta\over (\max(s,t)+ |x-w-z|)^{n+\eta}}.
 \end{align*}
 This, together with the fact that
 \begin{align*}
 \int_{|u|<s} {\max(s,t)^\eta\over (\max(s,t)+ |u-w|)^{n+\eta}}\Big(1+\frac{|u|}{s}\Big)^{\lambda} \, du
 \leq 2^{\lambda}\int_{|u|<s}  {\max(s,t)^\eta\over (\max(s,t)+ |u-w|)^{n+\eta}}\, du\leq C,
\end{align*}
shows
\begin{eqnarray}\label{e4.4}
&&\hspace{-1.5cm} \int_{\R} \big|K_{\psi(t\sqrt{L})
\Psi(s\sqrt{L})}(x-w,z)\big|\left(1+\frac{|x-z|}{s}\right)^{\lambda} \, dz \nonumber\\
 &\leq& C  \min\left( \left({s\over t}\right)^{2\kappa}, \left({t\over s}\right)^{2}\right)\left[1+
 \int_{|u|\geq s}{\max(s,t)^\eta\over (\max(s,t)+ |u-w|)^{n+\eta}}\left(1+\frac{|u|}{s}\right)^{\lambda} \, du\right].
\end{eqnarray}
To estimate the integrals over $|u|\geq s$, we note that if   $s\geq t$, then we use the
fact that $\eta>\lambda$ and $s+|u-w|\geq t+|u-w|\geq |w|+|u-w|\geq |u|$ to obtain
\begin{eqnarray}\label{e4.5}
 \int_{|u|\geq s} {s^\eta\over ({s}+ |u-w|)^{n+\eta}}\left(1+\frac{|u|}{s}\right)^{\lambda} \, du
 &\leq& 2^{\lambda}\int_{|u|\geq s} {s^\eta\over ({s}+ |u-w|)^{n+\eta}}\frac{|u|^{\lambda}}{s^\lambda}\, du\nonumber\\
 &\leq& C\int_{|u|\geq s} \frac{s^\eta} {{|u|}^{n+\eta}}\frac{|u|^{\lambda}}{s^\lambda}\, du
 \leq C.
\end{eqnarray}
If  $s<t$, then from the  fact that $t+|u-w|\geq |w|+|u-w|\geq |u|$  and $\eta>\lambda$,
\begin{eqnarray}\label{e4.6}
 \int_{|u|\geq s} {t^\eta\over ({t}+ |u-w|)^{n+\eta}}\left(1+\frac{|u|}{s}\right)^{\lambda} \, du
 &\leq& 2^{\lambda}\int_{|u|\geq s} {t^\eta\over ({t}+ |u-w|)^{n+\eta}}\frac{|u|^{\lambda}}{s^\lambda}\, du \nonumber\\
 &\leq& C\int_{|u|\geq s} \frac{t^\eta} {{|u|}^{n+\eta}}\frac{|u|^{\lambda}}{s^\lambda}\, du
 \leq C \left({t\over s}\right)^{\eta}.
\end{eqnarray}
Putting estimates \eqref{e4.5} and \eqref{e4.6} into \eqref{e4.4}, we have obtained that for any $|w|<t$,
\begin{eqnarray*}
 \int_{\R} \big|K_{\psi(t\sqrt{L})\Psi(s\sqrt{L})}(x-w,z)\big|\left (1+\frac{|x-z|}{s}\right)^{\lambda} \, dz
 &\leq&   C  \min\left( \left({s\over t}\right)^{2\kappa}, \left({s\over t}\right)^{2}\right)
 \left[1+ \max \left(1, \left({t\over s}\right)^{\eta} \right)\right]\\
 &\leq&   C \min\left( \left({s\over t}\right)^{2\kappa-\eta},  \left({t\over s}\right)^{2}\right).
\end{eqnarray*}
Observe that $\eta<2\kappa$. It follows
\begin{eqnarray*}
 \sup\limits_{|w|< t}\int_{\UHRN} \big|K_{\psi(t\sqrt{L})\Psi(s\sqrt{L})}(x-w,z)
 \big|\left(1+\frac{|x-z|}{s}\right)^{\lambda} \frac{dzds}{s}
 \leq C\int_0^\infty \min\left( \left({s\over t}\right)^{2\kappa-\eta},  \left({t\over s}\right)^{2}\right) \frac{ds}{s}\leq C,
\end{eqnarray*}
which shows estimate \eqref{e5}, and  the proof of Proposition~\ref{le3.1} is end.
\end{proof}

\bigskip

\noindent
{\it Proof of Theorem~\ref{th1.1}}.
To prove  the implication (ii) $\Rightarrow$   (i) of Theorem~\ref{th1.1},
from Proposition~\ref{le3.1} it suffices to show that for
 $f\in\HMAX\cap L^2({\Bbb R}^n),$    $f$ has a $(p, \infty, M)$
atomic representation.

We start with a suitable
version of the Calder\'on reproducing formula. Let $\Phi$ be a function
defined in Lemma~\ref{le2.1}, and set $\Psi(x):=x^{2M}\Phi(x)$,
$x\in{\mathbb{R}}$. By the spectral theory (\cite{Yo}), for
every $f\in L^2({\Bbb R}^n)$ one can write
 \begin{align}\label{e3.1}
f&=c_{\Psi}\int_0^\infty
\Psi({t\sqrt{L}})t^2Le^{-t^2L}f \, \frac{{\rm d}t}{t}\nonumber\\[4pt]
&=\lim_{\epsilon\to 0}c_{\Psi}\int_{\epsilon}^{1/\epsilon}
\Psi({t\sqrt{L}})t^2L e^{-t^2L}f \, \frac{{\rm d}t}{t}
\end{align}
with the integral converging in $L^2({\Bbb R}^n).$

Set
\begin{align*}
\eta(x)
:= c_{\Psi} \int_1^{\infty} t^2 x^2\Psi(tx) e^{-t^2x^2} \frac{dt}{t}
=
c_{\Psi} \int_x^{\infty} y\Psi(y)  e^{-y^2} dy, \quad x\neq 0
\end{align*}
with $\eta(0)=1$. It follows that  $\eta\in{\mathscr S}(\RR)$ is an even function, and
$$
\eta(a  x)-\eta(b  x)=c_{\Psi}\int_a^b t^2x^2\Psi(t  x)e^{-t^2 x^2} \frac{dt}{t}.
$$
By the spectral theory (\cite{Yo}) again, one has
\begin{align}\label{3.3}
c_{\Psi}\int_a^b
\Psi({t\sqrt{L}})t^2Le^{-t^2L}f \, \frac{{\rm d}t}{t}=\eta(a\sqrt{L})f(x)-\eta(b\sqrt{L})f(x).
\end{align}
Define,
$${\mathscr M}_Lf(x):=\sup\limits_{|x-y|<5\sqrt{n}t}\Big(|t^2Le^{-t^2L}f(y)|+|\eta(t\sqrt{L})f(y)|\Big).
$$
 By Proposition~\ref{le3.1},  it follows that
$$\|{\mathscr M}_Lf\|_{L^p(\R)}\leq C\|f\|_{\HMAX}, \ \ \ \ \  0<p\leq 1.
$$

Recall that ${\mathbb R}^{n+1}_+$  denotes the  upper
half-space in ${\mathbb R}^{n+1}$.
 If $O$ is an
open subset of ${\mathbb R}^n$, then the ``tent" over $O$, denoted
by ${\widehat O}$, is given as ${\widehat O}:=\{(x,t)\in \UHRN:  B(x, 4\sqrt{n}t)\subset O\}$.
For $ i\in{\Bbb Z}$, we  define the family of sets $O_i: =\{x\in {\Bbb R}^n: {\mathscr M}_Lf(x)>2^i\}$.
Now let $\{Q_{ij}\}_j$ be a  Whitney decomposition of  $O_i$  such that
$O_i=\cup_j Q_{ij}$ and let ${\widehat{O_i}}$ be a tent
region.
Set $\bar{e}=(1,\cdots,1)\in \R$.  For every $i,j$, we define
\begin{align}
\tilde{Q}_{ij}:=\{({y}, t)\in \UHRN: {y}+3t\bar{e}\in Q_{ij}\}.
\end{align}
It can be verified  that $\widehat{O_i}\subset \cup_j\tilde{Q}_{ij}$.
Indeed, for each $({y}^0, t^0)\in \widehat{O_i}$, we have that $B({y}^0,
4\sqrt{n}t^0)\subset O_i$.  Let  $\tilde{y}^0:={y}^0+3\bar{e}t^0$.
Observe that $\tilde{y}^0\in B({y}^0, 4\sqrt{n}t^0)$, and then $\tilde{y}^0\in O_i$.
 Then there exists some $Q_{ij_0}\subset O_i$ such that
$\tilde{y}^0\in Q_{ij_0}$,   hence   $({y}^0, t^0)\in \tilde{Q}_{ij_0}$
and $\widehat{O_i}\subset \bigcup_j\tilde{Q}_{ij}$.
Note that $\tilde{Q}_{ij}\cap \tilde{Q}_{ij'}=\emptyset$ when  $j\neq j'$.
We obtain an decomposition for $\UHRN$ as follows:
\begin{align*}
\UHRN=\cup_i \widehat{O}_i=\cup_i\widehat{O_i}
\big\backslash \widehat{O_{i+1}}=\cup_i\cup_j T_{ij},
\end{align*}
where
\begin{align*}
T_{ij}:=\tilde{Q}_{ij}\cap \widehat{O_i}
\big\backslash {\widehat{O_{i+\!1}}}.
\end{align*}
 Using the formula (\ref{e3.1}), one can write
 \begin{align}\label{e3.3}
f&=\sum_{ i, j }c_{\Psi}\int_0^\infty
\Psi(t\sqrt{L})\Big( \chi_{T_{ij}}t^2Le^{-t^2L}f\Big)\, \frac{{\rm d}t}{t}\nonumber\\[4pt]
&=: \sum_{i, j}  \lambda_{ij} a_{ij}
\end{align}
with the sum converging in $L^2(\R)$, where $\lambda_{ij}:= 2^{i}|Q_{ij}|^{1/p}, \    a_{ij}:=L^Mb_{ij}$, and
$$
b_{ij}:=(\lambda_{ij})^{-1} c_{\Psi}\int_0^\infty t^{2M} \Phi
(t\sqrt{L})\Big( \chi_{T_{ij}}t^2Le^{-t^2L}f\Big)\,
\frac{{\rm d}t}{t}.
$$

Let us show that the sum \eqref{e3.3} converges in $L^2(\R)$. Indeed, since
for $f\in L^2(\R)$,
$$\left(\int_{\UHRN} |t^2Le^{-t\sqrt{L}}f(y)|^2 \, \frac{dydt}{t}\right)^{1/2}\leq C \|f\|_{L^2(\R)},
$$
we use \eqref{e3.3} to obtain
 \begin{eqnarray*}
\left\|\sum_{|i|>N_1, |j|>N_2}  \lambda_{ij} a_{ij}\right\|_{L^2(\R)}&=&c_{\Psi}\left\|\sum_{|i|>N_1, |j|>N_2}\int_{\UHRN}K_{
(t^2L)^{M }\Phi({t\sqrt{L}})}(x,y) \chi_{T_{ij}}(y,t) t^2Le^{-t\sqrt{L}}f(y) \, \frac{{\rm d}t}{t}\right\|_{L^2(\R)} \\
&\leq& \sup\limits_{\|g\|_2\leq 1} \sum_{|i|>N_1, |j|>N_2}\int_{T_{ij}}
\big|(t^2L)^{M }\Phi({t\sqrt{L}})g(y) t^2Le^{-t\sqrt{L}}f(y)\big| \, \frac{dydt}{t}\\
&\leq &C\left(\sum_{|i|>N_1, |j|>N_2}\int_{T_{ij}}  |t^2Le^{-t\sqrt{L}}f(y)|^2 \, \frac{dydt}{t}\right)^{1/2} \to 0
\end{eqnarray*}
as $N_1\to \infty, N_2\to \infty.$

Next, we will show that, up to a normalization by a  multiplicative constant,
the $a_{ij}$ are $(p, \infty, M)$-atoms. Once the claim is established, we
shall have
\begin{align*}
\sum_{i,j}|\lambda_{ij}|^p = \sum_{i,j}2^{ip}|Q_{ij}|\leq C\sum_i2^{ip}|O_i|\leq C\|f\|^p_{\HMAX}
\end{align*}
as desired.

Let us now prove   that for every $i, j$, the function
$C^{-1} a_{ij}$
 is a $(p, \infty, M)$-atom associated with the cube $30Q_{ij}$ for some  constant $C$.
Observe that if $(y,t)\in T_{ij}$, then $B(y,4\sqrt{n}t)\in O_i$.  Denote by
$\tilde{y}:=y+3t\bar{e}$, and so $\tilde{y}\in Q_{ij}$ and  $B(\tilde{y},\sqrt{n}t)\in O_i$.
The fact that $Q_{ij}$ is the Whitney cube of $O_i$ implies that $5Q_{ij}\cap O_i^c\neq
\emptyset$.    Denote the side length of $Q_{ij}$ by $\ell(Q_{ij})$.   It then follows that
$t\leq 3 \ell(Q_{ij})$. Since $y+3\bar{e}t\in Q_{ij}$, we have that $y\in 20Q_{ij}$.
From Lemma~\ref{le2.1},   the integral kernel
 $K_{(t^{2}L)^k\Phi(t\sqrt{L})}$ of the operator $(t^{2}L)^k\Phi(t\sqrt{L})$ satisfies
$$
{\rm supp}\, K_{(t^{2}L)^k\Phi(t\sqrt{L})} \subseteq
\big\{(x,y)\in\mathbb{R}^n\times\mathbb{R}^n: |x-y|\leq t\big\}.
$$
\noindent This concludes that for every $k=0,1, \cdots, M$
 $$
{\rm supp}\, \big(L^k b_{ij}\big) \subseteq 30Q_{ij}.
$$
It remains to show that   $\big\|((\ell(Q_{ij})^2L)^{k}b_{ij}\big\|_{L^\infty({\Bbb R}^n)}\leq C (\ell(Q_{ij}))^{2M}
|Q_{ij}|^{-1/p},\ k=0,1, \cdots, M$. When  $K=0,1,\cdots, M-1$,
  it reduces to show
\begin{align}\label{e3.12}
\Big|\int_0^\infty\int_{\R} K_{
t^{2M}L^K\Phi(t^2L)}(x,y) \chi_{T_{ij}}(y,t)t^2Le^{-t^2L}f(y) \,dy\frac{{\rm d}t}{t}\Big|\leq C2^i \ell(Q_{ij})^{2(M-K)}.
\end{align}
Indeed,  If $\chi_{T_{ij}}(y,t)=1$,
then $(y,t)\in (\widehat{O_{i+1}})^c$, and so $B(y,4\sqrt{n}t)\cap (O_{i+1})^c\neq \emptyset$.
Let  $\bar{x}\in B(y,4\sqrt{n}t)\cap (O_{i+1})^c$. We have that $|t^2Le^{-t^2L}f(y)|\leq
{\mathscr M}_{L}f(\bar{x})\le 2^{i+1}$. By Lemma~\ref{le2.1},
\begin{align*}
&\Big|\int_0^\infty\int_{\R} K_{t^{2M}L^K\Phi(t^2L)}(x,y) \chi_{T_{ij}}(y,t)t^2Le^{-t^2L}f(y) \,dy\frac{{\rm d}t}{t}\Big|\\
&\leq C2^i \Big|\int_0^{c\ell(Q_{ij})} t^{2(M-K)}\int_{\R}
\big|K_{(t^2L)^K\Phi(t^2L)}(x,y)\big|  \,dy \frac{{\rm d}t}{t}\Big|\\
&\leq C2^i \int_0^{c\ell(Q_{ij})} t^{2(M-K)} \frac{{\rm d}t}{t}\\
&\leq C2^i\ell(Q_{ij})^{2(M-K)}
\end{align*}
since   $K=0,1,\cdots, M-1$.

Now we consider the case  $k=M$.
   The  proof is based on a modification of technique due to A. Calder\'on \cite{C}.
In this case, we need  to prove that for every $i,j$,
\begin{align}\label{e3.10}
\Big|\int_0^\infty\!\!\int_{\R}
K_{\Psi(t\sqrt{L})}(x,y) \chi_{T_{ij}}(y,t)t^2Le^{-t^2L}f(y) \,dy\frac{{\rm d}t}{t}\Big|\leq C2^{i}.
\end{align}
To show (\ref{e3.10}),   we fix $x$ and  let $d(x, Q_{ij})< 30\sqrt{n} {\ell}(Q_{ij})$.
We claim that the properties of the set defining  $\chi_{T_{ij}}(y,t)$ imply that
  there exist intervals $(0,b_0), (a_1,b_1), \cdots, (a_N,\infty)$, $0<b_0\leq
a_1<b_1\leq \cdots\leq a_N, \ 1\leq N \leq 2n+2$, such that,  for $l=0,1,\cdots, N-1$,
there holds  $a_{l+1}\leq 3^{2n+2}b_l$  and

\begin{itemize}
\item[(a)] $K_{\Psi(t\sqrt{L})}(x,y) \chi_{T_{ij}}(y,t)=0$ for $t>a_N$;

\item[(b)]    either $K_{\Psi(t\sqrt{L})}(x,y) \chi_{T_{ij}}(y,t)=0$   or $K_{\Psi(t\sqrt{L})}(x,y)
\chi_{T_{ij}}(y,t)=K_{\Psi(t\sqrt{L})}(x,y)$  for all $t\in (a_l, b_l);$

\item[(c)]  either $K_{\Psi(t\sqrt{L})}(x,y) \chi_{T_{ij}}(y,t)=0$   or $K_{\Psi(t\sqrt{L})}(x,y) \chi_{T_{ij}}(y,t)
=K_{\Psi(t\sqrt{L})}(x,y)$  for all  $t\in (0,b_0)$.
\end{itemize}

Assuming this claim for the moment, we observe that for  $d(x, Q_{ij})< 30\sqrt{n} {\ell}(Q_{ij})$, one can write
\begin{eqnarray}\label{e3.8}
 &&\hspace{-1cm}\int_0^\infty\int_{\R}
K_{\Psi(t\sqrt{L})}(x,y) \chi_{T_{ij}}(y,t)t^2Le^{-t^2L}f(y) \,dy\frac{{\rm d}t}{t}\\
&=&\left\{\int_0^{b_0}+\sum_{l=1}^{N-1}\int_{a_l}^{b_l}\right\}\int_{\R}
K_{\Psi(t\sqrt{L})}(x,y)\chi_{T_{ij}}(y,t)t^2Le^{-t^2L}f(y) \,dy\frac{{\rm d}t}{t} \nonumber\\
&+& \sum_{l=1}^{N-1}\int_{b_l}^{a_{l+1}} \int_{\R}
K_{\Psi(t\sqrt{L})}(x,y)\chi_{T_{ij}}(y,t)t^2Le^{-t^2L}f(y) \,dy\frac{{\rm d}t}{t} \nonumber\\
&=&I_1(x)+I_2(x).\nonumber
\end{eqnarray}
To estimate $I_1(x)$, we note that if $a_l\leq a<b\leq b_l$ or $0\leq a<b\leq b_0$, then  one has either
$$\int_{a}^{b}\int_{\R}
K_{\Psi(t\sqrt{L})}(x,y) \chi_{T_{ij}}(y,t)t^2Le^{-t^2L}f(y) \,dy\frac{{\rm d}t}{t}=0,
$$
or by \eqref{3.3},
\begin{eqnarray*}
 \int_{a}^{b}\int_{\R}
K_{\Psi(t\sqrt{L})}(x,y) \chi_{T_{ij}}(y,t)t^2Le^{-t^2L}f(y) \,dy\frac{{\rm d}t}{t}
 &=&\int_{a}^{b}
\Psi(t\sqrt{L})t^2Le^{-t^2L}f(x) \,\frac{{\rm d}t}{t}\\
&=&\eta(a\sqrt{L})f(x)-\eta(b\sqrt{L})f(x).
\end{eqnarray*}
Observe that for each $a\leq t\leq b$, if $|x-y|<t$, then $\chi_{T_{ij}}(y,t)=1$.
This tells us that $(y,t)\in (\widehat{O_{i+1}})^c$, hence $B(y,4\sqrt{n}t)\cap O_{i+1}\neq \emptyset$.
Assume that $\bar{x}\in B(y,4\sqrt{n}t)\cap (O_{i+1})^c$. From this, we have that $|x-\bar{x}|
\leq |x-y|+|y-\bar{x}|< 5\sqrt{n}t$ and ${\mathscr M}_{L}f(\bar{x})\leq 2^{i+1}$.
It implies that $|\eta(t\sqrt{L})f(x)|\leq {\mathscr M}_{L}f(\bar{x})\le C2^{i+1}$
for every $a\leq t\leq b$.
Therefore,  $|\eta(a\sqrt{L})f(x)|\le C2^{i+1}$ and 
$|\eta(b\sqrt{L})f(x)|\le C2^{i+1}$, and so $|I_1(x)|\leq C2^{i+1}$.

Consider   $I_2(x)$.   If $\chi_{T_{ij}}(y,t)=1$,
then $(y,t)\in (\widehat{O_{i+1}})^c$. Thus $B(y,4\sqrt{n}t)\cap (O_{i+1})^c\neq \emptyset$.
Assume that $\bar{x}\in B(y,4\sqrt{n}t)\cap (O_{i+1})^c$. We have that $|t^2Le^{-t^2L}f(y)|\leq
{\mathscr M}_{L}f(\bar{x})\le 2^{i+1}$. This, together with $a_{l+1}\leq cb_l$, implies that
\begin{eqnarray}\label{e3.11}
 \Big|\int_{b_l}^{a_{l+1}}\int_{\R}
K_{\Psi(t\sqrt{L})}(x,y) \chi_{T_{ij}}(y,t)t^2Le^{-t^2L}f(y) \,dy\frac{{\rm d}t}{t}\Big|
&\leq& 2^{i+1}\Big|\int_{b_l}^{cb_l}\int_{\R}
|K_{\Psi(t\sqrt{L})}(x,y)|  \,dy\frac{{\rm d}t}{t}\Big|\nonumber\\
&\leq& C 2^{i+1} \int_{b_l}^{cb_l}\frac{1}{t}\, dt\leq  C 2^{i+1},
\end{eqnarray}
which yields that $|I_2(x)|\leq C 2^{i+1}$.

Combining  (\ref{e3.8}) and (\ref{e3.11}), we  obtain (\ref{e3.10}). It follows that
$\|a_{ij}\|_{L^\infty}\leq C |Q_{ij}|^{-1/p}$. Up to a normalization by a  multiplicative constant,
the $a_{ij}$ are $(p, \infty, M)$-atoms.

It remains to   prove the claim (a), (b) and (c).  Note that
$
\chi_{T_{ij}}(y,t)=\chi_{\widehat{O_i}}(y,t)\cdot\chi_{(\widehat{O_{i+\!1}})^c}(y,t)\cdot\chi_{\tilde{Q}_{ij}}(y,t);
$
Assume that $Q_{ij}=\{(y_1,\cdots, y_n):  c_k\leq y_k\leq d_k, k=1,\cdots,n\}$. Then
\begin{eqnarray*}
\chi_{\tilde{Q}_{ij}}(y,t)&=&\prod_{l=1}^n
\chi_{\{c_l\leq y_l+3t\leq d_l\}}(y,t)  \\
&=&\prod_{l=1}^n \chi_{\{ y_l+3t\geq  c_l\}}(y,t)\cdot\chi_{\{ y_l+3t\leq d_l\}}(y,t).
\end{eqnarray*}
Let $\chi_l(y,t)$ be one of the characteristic functions $\chi_{\widehat{O_i}}(y,t),
 \chi_{(\widehat{O_{i+\!1}})^c}(y,t), \chi_{\{ y_l+3t\leq d_l\}}(y,t)$ and
 $\chi_{\{ y_l+3t\geq c_l\}}(y,t)$. We will prove that there exist  numbers $b_l$ and
$a_{l+1}$, $0<b_l\leq a_{l+1}, \ a_{l+1}\leq 3b_l$  such that
\begin{itemize}
 \item[{\bf(P)}]
  Given $x$, either $K_{\Psi(t\sqrt{L})}(x,y) \chi_{l}(y,t)=0$   or
 $K_{\Psi(t\sqrt{L})}(x,y) \chi_{l}(y,t)=K_{\Psi(t\sqrt{L})}(x,y)$  for all $t$ in each
 of the intervals complementary to $(b_l, a_{l+1})$. And for at least one of
 $\chi_l(y,t)$,    $K_{\Psi(t\sqrt{L})}(x,y) \chi_{l}(y,t)=0$ for $t>a_{l+1}$.
\end{itemize}
Then the same holds for $\chi_{T_{ij}}(y,t)= \prod_{l=1}^{2n+2}\chi_l(y,t)
K_{\Psi(t\sqrt{L})}(x,y)$ in each of the intervals complementary  to the union of
the intervals $(b_l, a_{l+1})$, which is what was asserted in the claim.
Thus we merely have to prove  {\bf (P)}. To do this,  we consider four cases.

\medskip

\noindent
{\bf Case 1:}    $\chi_l(y,t)=\chi_{\{ y_l+3t\geq  c_l\}}(y,t)$.

 \smallskip

 In this case, since supp $K_{\Psi(t\sqrt{L})}(x,y)\subseteq \{y: |x-y|\leq t\}$,
 we have that  supp $K_{\Psi(t\sqrt{L})}(x,y)\subseteq   \{y: x_l-t\leq y_l\leq x_l+t\}$.
If $x_l\geq c_l$, then  $y_l+3t\geq x_l+2t\geq c_l$ for any $t>0$. This yields
$$
K_{\Psi(t\sqrt{L})}(x,y) \chi_{l}(y,t)=K_{\Psi(t\sqrt{L})}(x,y),  \ \ \ \ t>0.
$$
If $x_l<c_l$, then we choose $b_l=\frac{c_l-x_l}{4}$ and $a_{l+1}= \frac{c_l-x_l}{2}$
(see Figure 1 below).

\smallskip

\begin{center}\label{eee}
\begin{tikzpicture}
\draw[thick,->,>=stealth](-2,0)--(6.5,0)node[right] {$\mathbb{R}^n$};
\draw[thick,->,>=stealth](0,0)--(0,4) node[above] {$t$};
\draw[domain=1:1,black,samples at={0,0,1.3}]plot (\x,{3*\x}) ;
\draw[domain=1:1,black,samples at={0,0,-1.3}]plot (\x,{-3*\x});
\draw[domain=1:1,black,samples at={2,-1.5,4}]plot (\x,{-0.5*(\x-4)});
\draw[domain=1:1,black,samples at={2,-1.5,6}]plot (\x,{-0.5*(\x-6)});
\draw[black,dashed](-0.57,1.71)--(0.57,1.71);
\draw[black,dashed](-0.8,2.4)--(0.8,2.4);
\fill[black!100] (0,1.71)  circle (0.4ex);
node\node[] at (0.2,1.5) {$b_l$};
\fill[black!100] (0,2.4)  circle (0.4ex);
node\node[] at (0.4,2.2) {$a_{l+1}$};
node\node[] at (3,1) {$\widetilde{Q}_{ij}$};
\coordinate (C) at(5,0);
\coordinate (B) at(4,0);
\draw (C)+(0,-0.15) node[yscale=1,xscale=6,rotate=90]{\{} ;
\draw (C)+(0,-0.15) coordinate[label=below:{$Q_{ij}$}] ;
\end{tikzpicture}
\end{center}
\begin{center}
Figure 1.
\end{center}

\medskip

\noindent
In the case of $t<b_l$,  we have $y_l+3t\leq x_l+4t<c_l,$
which implies that  $K_{\Psi(t\sqrt{L})}(x,y) \chi_{T_{ij}}(y,t)=0$. In the case of  $t> a_{l+1}$, we have $
y_l+3t\geq x_l+2t>c_l.$
This implies that $K_{\Psi(t\sqrt{L})}(x,y) \chi_{l}(y,t)=K_{\Psi(t\sqrt{L})}(x,y)$.

\medskip

\noindent
{\bf Case 2:}    $\chi_l(y,t)=\chi_{\{ y_l+3t\leq  d_l\}}(y,t)$.

\smallskip

 Since supp $K_{\Psi(t\sqrt{L})}(x,y)\subseteq \{y: |x-y|\leq t\}$,
 we have that  supp $K_{\Psi(t\sqrt{L})}(x,y)\subseteq  \{y: x_l-t\leq y_l\leq x_l+t\}$.
When $x_l\geq d_l$,  we have that $y_l+3t\geq x_l+2t> d_l$ for any $t>0$.  This tells us
$$
K_{\Psi(t\sqrt{L})}(x,y) \chi_{l}(y,t)=0,  \ \ \ {\rm for} \ t>0.
$$
When $x_l<d_l$, we choose $b_l=\frac{d_l-x_l}{4}$ and $a_{l+1}= \frac{d_l-x_l}{2}$. If $t<b_l$, then
 $
y_l+3t\leq x_l+4t<d_l,
$
which implies that
 $K_{\Psi(t\sqrt{L})}(x,y) \chi_{l}(y,t)=K_{\Psi(t\sqrt{L})}(x,y).
$
If  $t> a_{l+1}$, then
$
y_l+3t\geq x_l+2t>d_l.
$
From this, we have that
 $K_{\Psi(t\sqrt{L})}(x,y) \chi_{l}(y,t)=0.$

\medskip

\noindent
{\bf Case 3:}   $\chi_l(y,t)=\chi_{\widehat{O_i}}(y,t)$.

\smallskip

In this case, we choose $b_l=\frac{1}{5\sqrt{n}}d(x, O_i^c)$
and $a_{l+1}=\frac{1}{2\sqrt{n}}d(x, O_i^c)$. Let $|x-y|<t$. If $t< \frac{1}{5\sqrt{n}}d(x, O_i^c)$, then
$
d(y, O_i^c)\geq d(x,O_i^c)-|x-y|> 5\sqrt{n}t-t\geq 4\sqrt{n}t.
$
This tells us
$$K_{\Psi(t\sqrt{L})}(x,y) \chi_{l}(y,t)=K_{\Psi(t\sqrt{L})}(x,y)
$$
for   $t< \frac{1}{5\sqrt{n}}d(x, O_i^c)$.
If $t> \frac{1}{2\sqrt{n}}d(x, O_i^c)$, then
$
d(y, O_i^c)\leq  d(x,O_i^c)+d(x,y)< (2\sqrt{n}+1)t<4\sqrt{n}t.
$
Hence, if  $t>\frac{1}{2\sqrt{n}}d(x, O_i^c)$, then
$$K_{\Psi(t\sqrt{L})}(x,y) \chi_{l}(y,t)=0.
$$

\smallskip

\noindent
{\bf Case 4:}   $\chi_l(y,t)=\chi_{(\widehat{O_{i+\!1}})^c}(y,t)$.

\smallskip

In this case,
we can choose $b_l=\frac{1}{5\sqrt{n}}d(x, O_{i+1}^c)$ and $a_{l+1}=\frac{1}{2\sqrt{n}}d(x, O_{i+1}^c)$.
The proof can be an adaptation of the proof  as in {\bf Case 3}, and we omit the detail here.

\smallskip

This concludes the proof of the claim {\bf (P)}.
We have obtained the proof of
 an implication of  (ii) $\Rightarrow$ (i)  of Theorem~\ref{th1.1}.
 The proof of Theorem \ref{th1.1} is complete.
  \hfill{}$\Box$

  \bigskip

\bigskip

\bigskip

In the end of this section, we consider an electromagnetic Laplacian
$$
L=(i\nabla -A(x))^2 +V(x), \ \ \ \ n\geq 3.
$$
Recall that a measurable function $V$ on $\R$ is in the Kato class when
$$
 \lim_{r\downarrow 0}\sup_x \int_{|x-y|<r} {|V(y)|\over |x-y|^{n-2}} dy=0,
$$
while the Kato norm is defined by
$$
\|V\|_K=\sup_x   \int  {|V(y)|\over |x-y|^{n-2}} dy.
$$

\medskip

\begin{prop}\label{prop4.1}
  Consider an electromagnetic Laplacian
$$
L=(i\nabla -A(x))^2 +V(x), \ \ \ \ n\geq 3.
$$ Assume that  $A\in L^2_{\rm loc}(\R, \R)$,
and the positive and negative parts $V_{\pm}$ of $V$ satisfy $V_+$ is of Kato class,
$\|V_-\|_K <c_n=\pi^{n/2}/\Gamma(n/2-1)$.  Then for every number $M> {n\over 2}({1\over p}-1)$,
the spaces $ H^p_{L,  \max}(\RR^n)$ and $H^p_{L, {\rm at},  \infty, M}(\RR^n)$ coincide. In particular,
$$
\|f\|_{ H^p_{L,  \max}(\RR^n)}\approx\| f\|_{H^p_{L, {\rm at}, \infty, M}(\RR^n)}.
$$
\end{prop}

\begin{proof}
It is known (see \cite{CD}) that under assumptions of Proposition~\ref{prop4.1},
the operator $L$ has a unique nonnegative self-adjoint
extension, $e^{-tL}$ is an integral operator whose kernel satisfies the Gaussian estimate ({\bf H2}).
 Now  Proposition~\ref{prop4.1}
is a straightforward consequence of
 Theorem~\ref{th1.1}.
\end{proof}

 \bigskip

\noindent
{\bf Remarks.}\  i) Consider $L=-\Delta +V(x)$, where $V\in L^1_{\rm loc}({\Bbb R}^n)$ is a
non-negative function on ${\Bbb R}^n$. It is proved in \cite{HLMMY} that
the spaces $ H^1_{L,  \max}(\RR^n)$ and $H^1_{L, S}(\RR^n)$ are equivalent,
and then for every number $M\geq 1$,  $ H^1_{L,  \max}(\RR^n)$ and $H^1_{L, {\rm at}, 2, M}(\RR^n)$ coincide.
See also \cite{DZ}.
  However, the result of Proposition~\ref{prop4.1}  is  new.

\smallskip

ii)  Given an operator    $L$ satisfying ${\bf (H1)}$-${\bf (H2)}$,
we may define the spaces $H^p_{L,\ rad}(\R), 0<p\leq 1 $
as the completion of $\{f \in L^2(\R): \|f^+_L\|_{L^p(\R)}<
\infty\}$ with respect to $L^p$-norm of the radial  maximal function; i.e.,
\begin{eqnarray*}
\big\|f\big\|_{H^p_{L,  rad}(\R)}:=\big\|f^+_L\big\|_{L^p(\R)}:=\left\|\sup\limits_{t>0}|e^{-t^2L}f(x)|\right\|_{L^p(\R)} .
\end{eqnarray*}
For every $1<q\leq \infty$ and every
number $M> {n\over 2}({1\over p}-1)$, any $(p, q, M)$-atom $a$ is in $ H^p_{L, rad}(\R)$   and so the following
continuous inclusions hold:
$$
H^p_{L, {\rm at}, q, M}(\RR^n)\subseteq  H^p_{L,  rad}(\R).
$$
See \cite{HLMMY, DL, DZ}).
 To the best of our knowledge, the
continuous inclusion of  $ H^p_{L,   rad}(\R)\subseteq
H^p_{L, {\rm at}, q, M}(\RR^n)$
   is still an
open question.

\medskip

\bigskip

\noindent
{\bf Acknowledgments.} \  L. Song is supported by  NNSF of China (Grant No. 11471338) and the
Fundamental Research Funds for the Central Universities (Grant No. 14lgzd03).   L.  Yan is  supported by
 NNSF of China (Grant No.    11371378).
We would like to thank P. Auscher, P. Chen, X.T. Duong,  J. Li  and F. Negreira  for
useful discussions 
and comments.

\end{document}